\documentclass[a4paper]{amsart}
\usepackage{ifthen}
\usepackage[british]{babel}
\usepackage{amsmath}
\usepackage{amsthm}
\usepackage{amsfonts}

\renewcommand{\P}{\mathbb{P}}
\newcommand{\F}{\mathbb{F}}
\newcommand{\Q}{\mathbb{Q}}
\newcommand{\Z}{\mathbb{Z}}
\newcommand{\R}{\mathbb{R}}
\newcommand{\alg}{\mathcal{A}}
\DeclareMathOperator{\Br}{Br}
\DeclareMathOperator{\inv}{inv}
\newcommand{\leg}[2]{\genfrac{(}{)}{}{}{#1}{#2}}
\newcommand{\ad}{\mathbb{A}}
\DeclareMathOperator{\Pic}{Pic}
\renewcommand{\H}{\mathrm{H}}
\newcommand{\Xb}{\bar{X}}
\DeclareMathOperator{\Gal}{Gal}
\newcommand{\Gm}{\mathbb{G}_\textrm{m}}
\newcommand{\et}{\mathrm{\acute{e}t}}
\DeclareMathOperator{\Spec}{Spec}
\DeclareMathOperator{\im}{im}

\newtheorem{theorem}{Theorem}[section]
\newtheorem{lemma}[theorem]{Lemma}
\newtheorem{proposition}[theorem]{Proposition}

\theoremstyle{definition}

\theoremstyle{remark}
\newtheorem*{remark}{Remark}

\numberwithin{equation}{section}

\title{The Brauer--Manin obstruction on a general diagonal quartic surface}

\author[M. Bright]{Martin Bright}
\address{Mathematics Institute, University of Warwick, Coventry CV4
  7AL, UK}
\email{M.Bright@warwick.ac.uk}
\date{}

\begin{document}

\begin{abstract}
We show that, on a sufficiently general diagonal quartic surface,
there is no Brauer--Manin obstruction to the existence of rational
points.
\end{abstract}


\maketitle

\section{Introduction}

Our object of study is the diagonal quartic surface $X \subset
\P^3_\Q$ defined by the equation
\begin{equation} \label{eq:surface}
a_0 X_0^4 + a_1 X_1^4 + a_2 X_2^4 + a_3 X_3^4 = 0
\end{equation}
where $a_0,a_1,a_2,a_3 \in \Q$ are non-zero rational coefficients.

Multiplying the equation~\eqref{eq:surface} through by a constant,
permuting the coefficients, or changing any of the coefficients by a
fourth power gives rise to another equation defining a surface which
is clearly isomorphic (over $\Q$) to the original one.  Two diagonal
quartic equations related by such operations will be called
\emph{equivalent}.  In particular, after replacing $X$ with an
equivalent surface, we may assume that the coefficients $a_i$ are
integers with no common factor, and that none of them is divisible by
a fourth power.

When we talk about the \emph{reduction} of $X$ modulo some prime $p$,
we mean simply the variety in $\P^3_{\F_p}$ defined by reducing the
equation~\eqref{eq:surface} modulo $p$.  Suppose that $p$ is odd.
Then, according to the number of coefficients divisible by $p$, the
reduction at $p$ will be either: a smooth diagonal quartic surface; a
cone over a smooth diagonal quartic curve; (geometrically) a union of
four planes; or a quadruple plane.

\begin{theorem} \label{thm:main}
Let $X$ be the diagonal quartic surface over $\Q$ given
by~\eqref{eq:surface}, and let $H$ be the subgroup of $\Q^\times /
(\Q^\times)^4$ generated by $-1$, $4$ and the quotients $a_i / a_j$.
Suppose that the following conditions are satisfied:
\begin{enumerate}
\item \label{it:els} $X(\Q_v) \neq \emptyset$ for all places $v$ of $\Q$;
\item \label{it:235} $H \cap \{ 2,3,5 \} = \emptyset$;
\item \label{it:maximal} $|H| = 256$;
\item \label{it:reduction} there is some odd prime $p$ which divides
  precisely one of the coefficients $a_i$, and does so to an odd
  power; moreover, if $p \in \{ 7, 11, 17, 41 \}$, then the reduction
  of $X$ modulo $p$ is not equivalent to $x^4 + y^4 + z^4 = 0$.
\end{enumerate}
Then $\Br X / \Br \Q$ has order $2$, and there is no Brauer--Manin
obstruction to the existence of rational points on $X$.
\end{theorem}

\begin{remark}
It is easy to check that the group $H$ may also be generated by $-1$,
$4$ and $a_i/a_0$ $(i=1,2,3)$.  It follows that $H$ has order dividing
$256$.
\end{remark}

This theorem combines several ingredients, many of which are already
known.  The deepest part is the result, due to Ieronymou, Skorobogatov
and Zarhin~\cite{ISZ}, that condition~\ref{it:235} above implies the
vanishing of the transcendental part of the Brauer group of $X$,
meaning that $\Br X = \Br_1 X$.  The calculation that, under
condition~\ref{it:maximal} above, $\Br_1 X / \Br \Q$ has order $2$ can
be found in the tables contained in the author's
thesis~\cite{Bright:thesis}.  The new ingredients contained in this
article are a more geometric description of the non-trivial class of
Azumaya algebras on $X$ and a proof that these algebras, under
condition~\ref{it:reduction} above, give no obstruction to the
existence of rational points on $X$.

\subsection{Background}

Let us recall the definition of the Brauer--Manin obstruction; see
Skorobogatov's book~\cite{Skorobogatov:TRP} for more details.  Fix a
number field $k$ and a smooth, projective, geometrically irreducible
variety $X$ over $k$.  We define the \emph{Brauer group} of $X$ to be
$\Br X = \H^2_\et(X, \Gm)$.  If $K$ is any field containing $k$, and
$P \in X(K)$ a $K$-point of $X$, then there is an evaluation
homomorphism $\Br X \to \Br K$, $\alg \mapsto \alg(P)$, which is the
natural map coming from the morphism $P\colon \Spec K \to X$.  In
particular, this applies if $K=k_v$ is a completion of $k$.

As $X$ is projective, the set of adelic points of $X$ is $X(\ad_k) =
\prod_v X(k_v)$, the product being over all places of $k$.  The set
$X(\ad_k)$ is non-empty precisely when each $X(k_v)$ is non-empty,
that is, $X$ has points over every completion of $k$.  Let $\inv_v
\colon \Br k_v \to \Q/\Z$ be the invariant map.  Define the following
subset of the adelic points:
\[
X(\ad_k)^{\Br} = \big\{ (P_v) \in X(\ad_k) \  \big| \  \sum_v \inv_v \alg(P_v) = 0
\text{ for all } \alg \in \Br X \big\} \text{.}
\]
Suppose that $X(\ad_k)$ is non-empty.  By class field theory, the
diagonal image of $X(k)$ is contained in $X(\ad_k)^{\Br}$; if in fact
$X(\ad_k)^{\Br}$ is empty, then we say there is a \emph{Brauer--Manin
  obstruction} to the existence of $k$-rational points on $X$.

Let $\Xb$ denote the base change of $X$ to an algebraic closure
$\bar{k}$ of $k$.  There is a natural filtration on $\Br X$, given by
$\Br_0 X \subseteq \Br_1 X \subseteq \Br X$, where
\begin{itemize}
\item $\Br_0 X = \im( \Br k \to \Br X)$ consists of the constant
  classes in $\Br X$;
\item $\Br_1 X = \ker( \Br X \to \Br \Xb)$ is the \emph{algebraic}
  part of the Brauer group, consisting of those classes which are
  split by base change to $\bar{k}$.
\end{itemize}
If $X(\ad_k) \neq \emptyset$, then the natural homomorphism $\Br k \to
\Br X$ is injective, and we will think of $\Br k$ as being contained
in $\Br X$.  The elements of $\Br k$ do not contribute to the
Brauer--Manin obstruction, so in describing $X(\ad_k)^{\Br}$ it is
enough to consider the quotient $\Br X / \Br k$.

Elements of $\Br X \setminus \Br_1 X$ are called
\emph{transcendental}.  For certain types of varieties $X$, we know
that $\Br \Xb = 0$ and therefore that $\Br X$ is entirely algebraic:
this is true in particular if $X$ is a curve or a rational surface.
It is, however, certainly not true if $X$ is a K3 surface, such as our
diagonal quartic surface.  The transcendental part of the Brauer group
of a diagonal quartic surface has been studied by
Ieronymou~\cite{Ieronymou} and by Ieronymou, Skorobogatov and
Zarhin~\cite{ISZ}.  The algebraic part of the Brauer group of a
diagonal quartic surface has been studied by the present
author~\cite{Bright:thesis, Bright:JSC-2005}.

\subsection{Outline of the proof}

We will now describe an outline of the proof of Theorem~1.1, with the
details postponed to Section~\ref{sec:alg}.  As mentioned above, the
first ingredient is the following result of Ieronymou, Skorobogatov
and Zarhin.
\begin{theorem}[{\cite[Corollary~3.3]{ISZ}}]
Let $X$ and $H$ be as in the introduction, and suppose that $H \cap
\{2,3,5\} = \emptyset$.  Then $\Br X = \Br_1 X$.
\end{theorem}
So any Brauer--Manin obstruction on $X$ comes entirely from the
algebraic Brauer group.  The structure of $\Br_1 X / \Br \Q$ as an
abstract group can be computed using the isomorphism
\[
\Br_1 X / \Br \Q \cong \H^1(\Q, \Pic \Xb).
\]
In the case of diagonal quartic surfaces, $\Pic\Xb$ is generated by
the classes of the obvious 48 straight lines on $\Xb$, and
condition~\ref{it:maximal} of Theorem~\ref{thm:main} ensures that the
Galois action on these lines is the most general possible.
Lemma~\ref{lem:h1pic} below shows that $\Br_1 X / \Br \Q$ is of order
$2$.

It remains to compute the Brauer--Manin obstruction coming from the
non-trivial class in $\Br_1 X / \Br \Q$.  In Lemma~\ref{lem:alg}, we
describe explicitly an Azumaya algebra $\alg$ which may be defined on
any diagonal quartic surface~\eqref{eq:surface} for which $a_0 a_1 a_2
a_3$ is not a square.  Condition~\ref{it:maximal} implies in
particular that $a_0 a_1 a_2 a_3$ is non-square, so the algebra $\alg$
is defined on our particular surface.

The proof is completed by Lemma~\ref{lem:obs}.  This states that,
given a prime $p$ satisfying condition~\ref{it:reduction} of
Theorem~\ref{thm:main}, the Azumaya algebra $\alg$, evaluated at
different points of $X(\Q_p)$, gives invariants of both $0$ and
$\frac{1}{2}$.  In particular, $\alg$ is not equivalent to a constant
algebra, and provides no obstruction to the existence of rational
points on $X$.

\section{The algebraic Brauer--Manin obstruction}
\label{sec:alg}

In this section we describe an explicit Azumaya algebra on our
diagonal quartic surface.  For this purpose we may replace $\Q$ by an
arbitrary number field $k$.  Let $X \subset \P^3_k$ be the diagonal
quartic surface~\eqref{eq:surface}, and let $Y \subset \P^3_k$ be the
smooth quadric surface defined by
\[
a_0 Y_0^2 + a_1 Y_1^2 + a_2 Y_2^2 + a_3 Y_3^2 = 0 \text{.}
\]
There is a morphism $\phi\colon X \to Y$ given by $Y_i = X_i^2$.  If
$X$ is everywhere locally soluble, then so is $Y$; and, since $Y$ is a
quadric, it follows that $Y$ has a $k$-rational point.

\begin{lemma} \label{lem:alg}
Suppose that $X$ is everywhere locally soluble.  Pick a point
$P=[y_0,y_1,y_2,y_3] \in Y(k)$, and let $g \in k[Y_0,Y_1,Y_2,Y_3]$ be
the linear form
\[
g = a_0 y_0 Y_0 + a_1 y_1 Y_1 + a_2 y_2 Y_2 + a_3 y_3 Y_3
\]
defining the tangent plane to $Y$ at $P$.  Let
\begin{equation} \label{eq:fdef}
f = \phi^* g = a_0 y_0 X_0^2 + a_1 y_1 X_1^2 + a_2 y_2 X_2^2 + a_3 y_3 X_3^2
\end{equation}
be the quadratic form obtained by pulling $g$ back to $X$.  Write
$\theta = a_0 a_1 a_2 a_3$.  Then the quaternion algebra $\alg =
(\theta, f/X_0^2) \in \Br k(X)$ is an Azumaya algebra on $X$.  The
class of $\alg$ in $\Br X / \Br k$ is independent of the choice of
$P$.
\end{lemma}

\begin{remark}
Since $f$ is a quadratic form on $X$ and not a rational function, we
divide it by $X_0^2$ to obtain an element of $k(X)^\times$.  As
always, when defining a quaternion algebra over a field, $(a,b)$ and
$(a,bc^2)$ give isomorphic algebras.  So the choice of $X_0$ here is
completely arbitrary; we could replace it with any other $X_i$ or
indeed any linear form.
\end{remark}

\begin{remark}
The coordinates of the point $P$, and therefore the linear form $g$,
are only defined up to multiplication by a scalar.  So the point $P$
only determines the algebra $\alg$ up to an element of $\Br k$.
\end{remark}

\begin{proof}[Proof of Lemma~\ref{lem:alg}]
If $\theta$ is a square in $k$, then $\alg$ is isomorphic to the
algebra of $2\times 2$ matrices over $k(X)$, and the conclusions are
trivially true.  So suppose that $\theta$ is not a square in $k$.

As described for example in~\cite[Lemma~11]{SD:PLMS-2000}, to show
that $\alg$ is an Azumaya algebra we need to show that the principal
divisor $(f/X_0^2)$ is the norm of a divisor on $X$ defined over
$k(\sqrt{\theta})$.  Recall that, over $\bar{\Q}$, $\bar{Y}$ admits two
pencils of straight lines, the classes of which generate $\Pic \bar{Y}
\cong \Z^2$.  The tangent plane to $Y$ at $P$ intersects $Y$ in two
lines, $L$ and $L'$, which are each defined over $k(\sqrt{\theta})$
and conjugate over $k$; so the divisor of vanishing of $g$ on $Y$ is
$L+L'$.  Let $D = \phi^* L$ be the divisor obtained by pulling $L$
back to $X$, and similarly $D' = \phi^* L'$.  Then $(f/X_0^2) = D + D'
- 2 D_0 = N_{k(\sqrt{\theta})/k}(D - D_0)$, where $D_0$ is the divisor
on $X$ defined by $X_0=0$.

Independence of $P$ is a routine calculation, but we reproduce it for
the sake of completeness.  Let $P_1 \in Y(k)$ be another point, and
let $g_1$ be the corresponding linear form defining the tangent plane
to $Y$ at $P_1$.  Then the divisor of vanishing of $g_1$ on $Y$ is
$L_1 + L_1'$, where $L_1$ is a line, defined over $k(\sqrt{\theta})$
and linearly equivalent to $L$, and $L_1'$ its conjugate over $k$.  So
there exists a rational function $h$ on $Y$, defined over
$k(\sqrt{\theta})$, such that $(h) = L - L_1$.  Then
\[
(g_1N_{k(\sqrt{\theta})/k}h) = (L_1 + L_1') + (L - L_1) + (L' - L_1')
= L + L'
\]
and so $g_1 N_{k(\sqrt{\theta}/k)} h$ is a constant multiple of $g$.
Let $f_1 = \phi^* g_1$; then $f/f_1$ is a constant multiplied by the
norm of a rational function defined over $k(\sqrt{\theta})$, so
$(\theta, f_1/X_0^2)$ differs from $(\theta, f/X_0^2)$ only by a
constant algebra.
\end{proof}

\begin{remark}
Even when $\theta$ is not a square in $k$, it is still possible for
the class of $\alg$ in $\Br X / \Br k$ to be trivial.  For example,
taking $k=\Q$, the tables of~\cite{Bright:thesis} show that, for any
integers $c_1, c_2$, the diagonal quartic surface
\[
X_0^4 + c_1 X_1^4 + c_2 X_2^4 - c_1^2 c_2^2 X_3^4
\]
has $\Br_1 X = \Br \Q$.  In particular, the algebra $\alg$ on this
surface is equivalent to a constant algebra.  Note that
Lemma~\ref{lem:obs} below does not apply in this case, since no prime
divides exactly one of the coefficients.
\end{remark}

\begin{lemma}\label{lem:h1pic}
Let $X$ be a diagonal quartic surface over $\Q$.  In the notation of
Theorem~\ref{thm:main}, suppose that $|H|=256$.  Then $\Br_1 X / \Br
\Q$ is of order $2$.
\end{lemma}

\begin{proof}
This calculation can be found in~\cite{Bright:thesis}, and depends on
the well-known isomorphism $\Br_1 X / \Br \Q \cong \H^1(\Q, \Pic\Xb)$.
What follows is a brief summary of the calculation.  The variety $X$
contains (at least) 48 straight lines: for example, setting
\[
a_0 X_0^4 + a_1 X_1^4 = 0, \qquad a_2 X_2^4 + a_3 X_3^4 = 0
\]
and factorising each side over $\bar{\Q}$ gives equations for 16 lines;
the other 32 are obtained by permuting the indices.  The lines are all
defined over the extension $K = \Q(i, \sqrt{2}, \sqrt[4]{a_1/a_0},
\sqrt[4]{a_2/a_0}, \sqrt[4]{a_3/a_0})$.  The classes of the lines
generate the Picard group of $X$ over $\bar{\Q}$, which is free of
rank $20$.  By the inflation-restriction exact sequence, we have
$\H^1(\Q, \Pic\Xb) = \H^1(K/\Q, \Pic X_K)$ and computing this
cohomology group comes down to knowing the Galois group $\Gal(K/\Q)$
and its action on the 48 lines.  Appendix~A of~\cite{Bright:thesis}
lists the result of this computation for all possible Galois groups
$\Gal(K/\Q)$.  In particular, case A222 there is where $K/\Q$ is of
maximal degree 256, so that the coefficients $a_i$ are ``as general as
possible''.  In that case, $\H^1(K/\Q, \Pic X_K)$ is computed to be of
order $2$.

We claim that $[K:\Q] = |H|$, so that condition~\ref{it:maximal} of
Theorem~1.1 implies that $X$ falls into case A222
of~\cite{Bright:thesis}.  Kummer theory shows that $[K:\Q(i)] = |H'|$,
where $H'$ is the subgroup of $\Q(i)^\times / (\Q(i)^\times)^4$
generated by $4$ and the $a_j/a_0$.  The kernel of the natural map $r \colon
\Q^\times / (\Q^\times)^4 \to \Q(i)^\times / (\Q(i)^\times)^4$ is of
order $2$, generated by the class of $-4$; so $H' = r(H)$, and
\[
[K:\Q] = [K:\Q(i)][\Q(i):\Q] = 2|H'| = |H| \text{.}
\]
\end{proof}

\begin{remark}
Further cohomology calculations could show that, under the hypothesis
that $|H|=256$, the algebra $\alg$ of Lemma~\ref{lem:alg} represents
the non-trivial class in $\Br X / \Br \Q$.  However, there is no need
for this, since in our situation non-triviality is also implied by the
following lemma.
\end{remark}

\begin{lemma} \label{lem:obs}
Let $X$ be a diagonal quartic surface over $\Q$ given by
equation~\eqref{eq:surface}.  Let $\alg$ be the Azumaya algebra
described in Lemma~\ref{lem:alg}.  Suppose that $p$ is an odd prime such
that:
\begin{enumerate}
\item \label{it:cone} $p$ divides precisely one of the coefficients
  $a_0, a_1, a_2, a_3$, and does so to an odd power;
\item \label{it:locsol} $X(\Q_p)$ is not empty;
\item \label{it:special} if $p \in \{ 7, 11, 17, 41 \}$, then the
  reduction of $X$ modulo $p$ is not equivalent to the cone over the
  quartic curve $x^4 + y^4 + z^4 = 0$.
\end{enumerate}
Then $\inv_p \alg(Q)$ takes both values $0$ and $\frac{1}{2}$ for $Q
\in X(\Q_p)$.  In particular, the class of $\alg$ in $\Br X / \Br \Q$
is non-trivial, and $\alg$ gives no Brauer--Manin obstruction to the
existence of rational points on $X$.
\end{lemma}

\begin{remark}
Condition~\ref{it:cone} implies, in particular, that $\theta = a_0 a_1
a_2 a_3$ is not a square.
\end{remark}

\begin{remark}
Condition~(\ref{it:locsol}), that $X(\Q_p)$ be non-empty, is automatic
for $p \ge 37$: for the reduction of $X$ modulo $p$ is a cone over a
smooth quartic curve, which has a rational point by the Hasse--Weil
bound.  For $p < 37$, one can easily check by a computer search that
the only smooth diagonal quartic curves over $\F_p$ lacking a rational
point are the following (up to equivalence):
\begin{itemize}
\item $x^4 + y^4 + z^4 = 0$ for $p=5$ or $29$;
\item $x^4 + y^4 + 2z^4 = 0$ for $p=5$ or $13$.
\end{itemize}
\end{remark}

\begin{proof}
Suppose, without loss of generality, that $p \mid a_0$.

In constructing $\alg$, as described in Lemma~\ref{lem:alg}, we may
choose any point $P \in Y(\Q)$ to start from.  In particular, we may
choose $P$ such that $y_1, y_2, y_3$ are not all divisible by $p$, for
the following reason.  Recall that we have assumed the coefficients
$a_i$ to be fourth-power-free, so that in particular $v_p(a_0) \le 3$.
The original surface $X$ is locally soluble at $p$, so let $[x_0, x_1,
  x_2, x_3] \in X(\Q_p)$ with the $x_i$ $p$-adic integers, not all
divisible by $p$.  If $x_1, x_2, x_3$ were all divisible by $p$, then
we would have $v_p( a_1 x_1^4 + a_2 x_2^4 + a_3 x_3^4 ) \ge 4$,
whereas $v_p(a_0 x_0^4) \le 3$, and so the defining
equation~\eqref{eq:surface} could not be satisfied.  Now $[ x_0^2,
  x_1^2, x_2^2, x_3^2]$ is a point of $Y(\Q_p)$, with $x_1, x_2, x_3$
not all divisible by $p$, and so by weak approximation $Y(\Q)$
contains a point with the desired property.

Looking at the equation of $Y$ shows that, in fact, at most one of
$y_1, y_2, y_3$ can be divisible by $p$.  It would clarify the rest of
the argument if none of $y_1, y_2, y_3$ were divisible by $p$, and the
reader is encouraged to imagine this to be the case; but unfortunately
if $p=3$ it is not always possible.

Starting from such a $P$, we obtain $f$ as in~\eqref{eq:fdef} where
the coefficient of $X_0^2$ is divisible by $p$, but at least one of
the other coefficients is not divisible by $p$.  The reduction
$\tilde{f}$ of $f$ modulo $p$ is a non-zero diagonal quadratic form on
$\P^3_{\F_p}$, with no term in $X_0^2$.

We now reduce to a problem over $\F_p$.  Let $\tilde{X}$ denote the
reduction of $X$ modulo $p$.  Let $\tilde{Q} \in \tilde{X}(\F_p)$ be a
smooth point; then, by Hensel's Lemma, $\tilde{Q}$ lifts to a point $Q
\in X(\Q_p)$.  Suppose that $\tilde{f}(\tilde{Q}) \neq 0$, and that
$X_0(Q) \neq 0$.  Since $p$ divides $\theta$ to an odd power, the
description of the Hilbert symbol at~\cite[III, Theorem~1]{Serre:CA}
gives
\begin{equation} \label{eq:legendre}
\inv_p \alg(Q) = (\theta, f(Q)/X_0^2)_p = (\theta, f(Q))_p
= \leg{\tilde{f}(\tilde{Q})}{p} \text{.}
\end{equation}
Here the leftmost equality is abusing notation slightly, since
$\inv_p$ traditionally takes values in $\{0, \frac{1}{2}\}$ whereas
the Hilbert symbol $(\cdot, \cdot)_p$ takes values in $\{\pm 1\}$.
Since $f$ is of degree $2$, the value $f(Q)$ is defined only up to
squares, and likewise $\tilde{f}(\tilde{Q})$, but the expressions
in~\eqref{eq:legendre} are well defined.  The requirement that $X_0(Q)
\neq 0$ is superfluous, since we can always replace $\alg$ by the
isomorphic algebra $(\theta, f/X_i^2)$ for some $i \neq 0$ to show
that the conclusion of~\eqref{eq:legendre} still holds.

Now let $C$ be the smooth quartic curve in $\P^2_{\F_p}$ defined by
\[
C : \tilde{a}_1 X_1^4 + \tilde{a}_2 X_2^4 + \tilde{a}_3 X_3^4 = 0
\text{.}
\]
This is, of course, the same as the defining equation of $\tilde{X}$,
but now considered as an equation in only three variables.  Any point
of $X(\Q_p)$ reduces to give us a point of $\tilde{X}(\F_p)$ and
hence, forgetting the $X_0$-coordinate, of $C(\F_p)$.  Since the
diagonal quadratic form $\tilde{f}$ has no term in $X_0^2$, we can
consider it as a form on $C$.  Note that $\tilde{f}$ depends only on
$C$, not on our original variety $X$, since we may also construct
$\tilde{f}$ as follows: the point $\tilde{P} = (\tilde{y}_1,
\tilde{y}_2, \tilde{y}_3)$ lies on the smooth plane conic $Z:
\tilde{a}_1 Y_1^2 + \tilde{a}_2 Y_2^2 + \tilde{a}_3 Y_3^2 = 0$, and
the linear form $\tilde{g}$ defines the tangent line to $Z$ at
$\tilde{P}$.  Write $\tilde\phi$ for the map from $C$ to $Z$ given by
$Y_i = X_i^2$; pulling $\tilde{g}$ back under $\tilde\phi$ gives the
form $\tilde{f}$.  In particular, this shows that the divisor of
$\tilde{f}$ is a multiple of $2$: for we have $(\tilde{g}) = 2
\tilde{P}$ and therefore $(\tilde{f}) = 2( \tilde\phi^* \tilde{P} )$.
The geometric picture (which is only accurate as long as none of $y_1,
y_2, y_3$ are divisible by $p$) is that $\tilde{f}$ defines a plane
conic which is tangent to $C$ at four distinct points, which are the
four points mapping to $\tilde{P}$ under $\tilde\phi$.

Note also that the divisor $(\tilde{f})/2 = \tilde\phi^* \tilde{P}$ is
not a plane section: as long as none of $y_1, y_2, y_3$ are divisible
by $p$, this divisor consists of four distinct points of the form
$[\pm \alpha, \pm \beta, \pm \gamma]$, with $\alpha, \beta, \gamma$
all non-zero; in characteristic $\neq 2$, such points can never be
collinear.  If one of $y_1, y_2, y_3$ is divisible by $p$, then we
move to an extension of $\F_p$, replace $\tilde{P}$ by some
$\tilde{P}'$ for which the above proof does work, and observe that
$\tilde{P}'$ is linearly equivalent to $\tilde{P}$, so $\tilde\phi^*
\tilde{P}'$ is linearly equivalent to $\tilde\phi^* \tilde{P}$, but
$\tilde\phi^* \tilde{P}'$ is not a plane section; therefore neither
can $\tilde\phi^* \tilde{P}$ be a plane section.

By~\eqref{eq:legendre}, it remains to show that the quadratic form
$\tilde{f}$ takes both square and non-square non-zero values on
$C(\F_p)$.  Equivalently, we need to show that, for any $c \in
\F_p^\times / (\F_p^\times)^2$, the equations
\begin{equation} \label{eq:cover}
T^2 = c \tilde{f}(X_1, X_2, X_3), \qquad 
\tilde{a}_1 X_1^4 + \tilde{a}_2 X_2^4 + \tilde{a}_3 X_3^4 = 0
\end{equation}
have simultaneous solutions with $T$ non-zero.  These equations define
a double cover $E_c$ of $C$.  As given, $E_c$ is singular at the
points with $T=0$ (which are the points lying over the zeros of
$\tilde{f}$), so we consider its normalisation $E'_c \to E_c$.  This
is a smooth double cover of $C$ with the following properties:
\begin{itemize}
\item the morphism $E'_c \to E_c$ is an isomorphism outside $8$
  (geometric) points lying over the points of $E_c$ with $T=0$;
\item since the divisor $(\tilde{f})$ is a multiple of $2$, the
  quadratic extension of function fields $\F_p(E_c)/\F_p(C)$ is
  unramified and hence so is $E'_c \to C$;
\item since the divisor $(\tilde{f})/2$ is not a plane section, this
  extension contains no non-trivial extension of $\F_p$ and so $E'_c$
  is geometrically irreducible.
\end{itemize}

By the Riemann--Hurwitz formula, $E'_c$ has genus $5$.  If $p > 114$,
then the Hasse--Weil bounds show that $E'_c$ has strictly more than
$8$ points over $\F_p$, and so $E_c$ has at least one point with $T
\neq 0$, completing the argument in this case.

It remains to check the cases with $p < 114$.  For each prime $p$, we
can take $\tilde{a}_1 = 1$ and let $\tilde{a}_2, \tilde{a}_3$ run
through $\F_p^\times / (\F_p^\times)^4$.  A straightforward computer
search shows that the only cases when some $E_c$ fails to have points
are those listed in the statement of the lemma.
\end{proof}

\begin{remark}
With a slightly longer argument, we could avoid having to throw away
the points on $E_c$ with $T=0$.  By taking two different $\tilde{P}$s
to start with, we obtain two different $\tilde{f}$s with no common
zeros on $C$.  The ratio of the $\tilde{f}$s is a square, and the
corresponding equations~\eqref{eq:cover} patch together to give a
description of $E'_c$ with no singularities.  The sophisticated reader
will recognise $E'_c$ as a torsor under $\boldsymbol{\mu}_2$
corresponding to the class $(\tilde{f}/X_0^2) \in \Pic C[2]$.
\end{remark}

\section{A counterexample}

In this section we present a counterexample showing that
Theorem~\ref{thm:main} can fail when condition~\ref{it:reduction} is
not met.  We begin by giving an infinite family of diagonal quartics
satisfying conditions~\ref{it:els}--\ref{it:maximal} of
Theorem~\ref{thm:main}, but not condition~\ref{it:reduction}.

\begin{lemma} \label{lem:eg1}
Let $p,q$ be odd primes satisfying the following properties:
\begin{itemize}
\item $p \equiv q \equiv 3 \pmod 4$;
\item $p$ and $q$ are both fourth powers modulo $17$;
\item $\leg{p}{q}=1$.
\end{itemize}
Then the diagonal quartic surface
\[
X_0^4 + q X_1^4 = p X_2^4 + 17 pq X_3^4
\]
satisfies conditions~\ref{it:els}--\ref{it:maximal} of
Theorem~\ref{thm:main}.
\end{lemma}
\begin{proof}
Conditions~\ref{it:235} and~\ref{it:maximal} are clear, since there
are no non-obvious relations between the generators for $H = \langle
-1, 4, p, q, 17 \rangle$.  It remains to prove local solubility.  For
$\mathbb{R}$ this is clear.  For primes $\ell \ge 23$ of good
reduction, the Weil conjectures guarantee a point modulo $\ell$ and
hence a point over $\Q_\ell$ by Hensel's Lemma.  At
$\ell=3,7,11,13,17,19$, a computer search shows that every smooth
diagonal quartic surface has a rational point modulo $\ell$.  At
$\ell=5$, the only smooth diagonal quartic surface lacking a point
over $\F_5$ is the Fermat quartic $X_0^4 + X_1^4 + X_2^4 + X_3^4 = 0$,
so for local solubility to fail we would need $q \equiv -p \equiv
-17pq \equiv 1 \pmod 5$, which is impossible.

Since $p$ and $q$ are both congruent to $3 \pmod 4$, the fourth powers
modulo $p$ or $q$ are the same as the squares.  At $q$, the condition
$\leg{p}{q}=1$ guarantees local solubility; at $p$, we have
$\leg{-q}{p} = -\leg{q}{p} = \leg{p}{q}=1$ and so again the surface is
locally soluble.  Finally, at $17$, the fact that $p$, hence $-p$, and
$q$ are fourth powers means that the reduction at $17$ is isomorphic
to the cone over the Fermat quartic curve $x^4+y^4+z^4=0$, which has
smooth points over $\F_{17}$.
\end{proof}

However, choosing $p$ and $q$ to be fourth powers modulo $17$ means
that condition~\ref{it:reduction} of Theorem~\ref{thm:main} is not
satisfied.

We will show that the Azumaya algebra $\alg$ described in
Section~\ref{sec:alg} can give an obstruction to the existence of
rational points on $X$, at least for some values of $p$ and $q$.
Recall that multiplying the form $f$ by a constant changes $\alg$ by a
constant algebra.  To avoid contributions at unnecessary primes, we
choose our representation $\alg = (\theta,f)$ such that the
coefficients of $f$ are integers with no common factor.  (This is
equivalent to writing our point $P=[y_0, y_1, y_2, y_3]$ with the
$y_i$ integers having no common factor.)

\begin{lemma} \label{lem:eg2}
Let $X$ be the surface of Lemma~\ref{lem:eg1}, and let $\alg =
(\theta,f)$ be normalised as above.  Then, for all places $v \neq 17$,
$\inv_v \alg(Q)=0$ for $Q \in X(\Q_v)$.  The invariant is constant on
$X(\Q_{17})$.
\end{lemma}
\begin{proof}
Our normalisation of $f$ ensures that, at all places of good reduction
for $X$, the algebra $\alg$ also has good reduction and so the
invariant is zero at these places.
See~\cite[Corollary~4]{Bright:MPCPS-2007} for one explanation of why
this is true.

The primes of bad reduction for $X$ are $2,17,p,q$.  Observe that
$\theta = 17p^2q^2$ is a square in $\R$, $\Q_2$, $\Q_p$ and $\Q_q$
(the last two follow by quadratic reciprocity from the fact that $-p$
and $q$ are fourth powers, hence squares, modulo $17$).  So the
conclusion is true at each of these places.  Our normalisation of
$f$ ensures that, at all places of good reduction for $X$, the algebra
$\alg$ also has good reduction and so the invariant is zero at these
places.

At $17$, the argument used in the proof of Lemma~\ref{lem:obs} shows
that $\inv_{17} \alg(Q)$ is constant for $Q \in X(\Q_v)$.  We give the
details.  The reduction of $X$ modulo $17$ is isomorphic to the cone
$X_0^4 + X_1^4 + X_2^4 = 0$.  The corresponding quadric is $Y_0^2 +
Y_1^2 + Y_2^2=0$.  To show simply that the invariant is constant, we
can change $\alg$ by a constant algebra and so may as well replace $P$
by any point which is convenient.  So pick $\tilde{P}=[5,5,1,0]$ and
hence $\tilde{f} = 5 X_0^2 + 5 X_1^2 + X_2^2$.  (This choice of
$\tilde{P}$ has the advantage that it does not lift to a point of the
quartic, so $\tilde{f}$ is never zero on $\F_{17}$-rational points of
the quartic.)  Now the solutions to $X_0^4 + X_1^4 + X_2^4$ over
$\F_{17}$ are all of the form $[\epsilon,1,0]$, $[\epsilon,0,1]$ or
$[0,\epsilon,1]$ where $\epsilon^4=-1$, and it turns out that
evaluating $\tilde{f}$ at any of these points gives a square in
$\F_{17}$.
\end{proof}

We do not yet know whether the invariant at $17$ will be $0$ or
$\frac{1}{2}$.  If is it $0$, then there is no Brauer--Manin
obstruction on $X$ (not even to weak approximation).  If it is
$\frac{1}{2}$, then there is a Brauer--Manin obstruction to the
existence of rational points.  To determine which, we only need to
evaluate the invariant at one point.  A simple calculation reveals
that the first example satisfying the conditions of
Lemma~\ref{lem:eg1} does indeed give a counterexample to the Hasse
principle:

\begin{proposition}
Let $X$ be the diagonal quartic surface given by
\[
X_0^4 + 47 X_1^4 = 103 X_2^4 + (17\times 47\times 103) X_3^4 .
\]
Then $X$ has points in each completion of $\Q$, but the algebra $\alg$
gives a Brauer--Manin obstruction to the existence of a rational point
on $X$.
\end{proposition}

\begin{proof}
Since $X$ satisfies the conditions of Lemma~\ref{lem:eg1}, it only
remains to evaluate the obstruction at $17$.  On the quadric 
\[ 
Y : Y_0^2 + 47 Y_1^2 = 103 Y_2^2 + (17\times 47\times 103) Y_3^2,
\]
we can take $P=[20:13:-9:0] \in Y(\Q)$, and so obtain the Azumaya
algebra
\[
\alg = (17, (20 X_0^2 + (47\times 13) X_1^2 + (103\times 9)
X_2^2)/X_0^2).
\]
Evaluating the quadratic form $20 X_0^2 + (47\times 13) X_1^2 +
(103\times 9) X_2^2$ at any point of $X(\F_{17})$ gives a non-square
value modulo $17$, and therefore $\inv_{17} \alg(Q) = \frac{1}{2}$ for
all $Q \in X(\Q_{17})$.  Combined with the fact that the invariant is
$0$ at each other place, we deduce that $\sum_v \inv_v \alg(Q_v) =
\frac{1}{2}$ for all $(Q_v) \in X(\ad_\Q)$, and therefore that $\alg$
gives a Brauer--Manin obstruction to the existence of a rational point
on $X$.
\end{proof}

\begin{remark}
It was not \emph{a priori} clear that starting from different points
$P$ of the quadric $Y$ should always give the same invariant at $17$.
Different points $P$ might give algebras $\alg$ differing by a
constant algebra.  After all, in performing the verification of
Lemma~\ref{lem:eg2}, we could have replaced the point
$\tilde{P}=[5,5,1,0]$ by a scalar multiple, say $[1,1,7,0]$, and
$\tilde{f}$ would have been non-square at all points instead of
square.  However, our insistence that $P$ should be given by
coordinates which are coprime integers fixes the invariants at all
places other that $17$, and therefore (by the product rule) fixes the
invariant at $17$ as well.  A somewhat surprising conclusion is this:
that, given any point $\tilde{P} \in \tilde{Y}(\F_{17})$, at most half
of the scalar multiples of $\tilde{P}$ lift to rational points of $Y$
with coprime integer coordinates.
\end{remark}

\subsection*{Acknowledgements}
I thank Sir~Peter Swinnerton-Dyer for many useful conversations on
this subject, and Damiano Testa and Samir Siksek for comments on this
article.

\bibliographystyle{abbrv}
\bibliography{martin}

\end{document}